\documentclass[10pt,a4paper]{article}
\usepackage{amsmath,amssymb,amsthm,url}
\usepackage{amsfonts,enumerate,enumitem}
\usepackage{graphicx}

\usepackage[normalem]{ulem} 

\usepackage[usenames,dvipsnames]{color}

\newcommand\red[1]{#1}

\newif\ifcmts
\cmtstrue
\newcounter{sideremark}

\input epsf
\input{xy}\xyoption{all}

\usepackage[a4paper]{geometry}

\def\Z{{\mathbb Z}} \def\R{{\mathbb R}}  
\long\def\comment#1\endcomment{}

\def\id{\mathop{\fam0 id}}
\def\lk{\mathop{\fam0 lk}}
\def\Int{\mathaccent"7017 }
\def\im{\mathop{\fam0 im}}

\newcommand{\boundary}{\ensuremath{\partial}}

\newcommand{\iprod}{\raisebox{-0.5ex}{\scalebox{1.8}{$\cdot$}}\,}

\theoremstyle{plain}
\newtheorem{Theorem}{Theorem}[section]
\newtheorem{Lemma}[Theorem]{Lemma}

\newtheorem{Proposition}[Theorem]{Proposition}

\theoremstyle{definition}
\newtheorem{Remark}[Theorem]{Remark}

\newcommand{\define}[1]{\textbf{#1}}

\begin{document}

\title{Eliminating Higher-Multiplicity Intersections, III.\\ Codimension 2
\thanks{Research supported by the Swiss National Science Foundation (Project SNSF-PP00P2-138948), by the Russian Foundation for Basic Research Grant No. 15-01-06302, by Simons-IUM Fellowship and by the D. Zimin's Dynasty Foundation Grant.
\newline
We would like to thank A. Klyachko, V. Krushkal, S. Melikhov, M. Tancer, P. Teichner and anonymous referees for helpful discussions.}
}

\def\istsymb{\textrm{a}}
\def\moscowsymb{\textrm{b}}

\author{S. Avvakumov$^{\istsymb}$ \and I. Mabillard$^{\istsymb}$ \and A. Skopenkov$^{\moscowsymb}$
\and U. Wagner$^{\istsymb}$}

\date{}

\maketitle

{\renewcommand\thefootnote{\istsymb}
\footnotetext{IST Austria, Am Campus 1, 3400 Klosterneuburg, Austria. Email: \texttt{savvakum@ist.ac.at, imabillard@ist.ac.at, uli@ist.ac.at}}
}

{\renewcommand\thefootnote{\moscowsymb}
\footnotetext{Moscow Institute of Physics and Technology,
Institutskiy per., Dolgoprudnyi, 141700, Russia, and Independent University of Moscow.
B. Vlasy\-ev\-skiy, 11, Moscow, 119002, Russia. Email: \texttt{skopenko@mccme.ru}}
}

\vspace{-3ex}
\begin{abstract}
\noindent We study conditions under which a finite simplicial complex $K$ can be mapped to
$\R^d$ without higher-multiplicity intersections.
An \emph{almost $r$-embedding} is a map $f\colon K\to \R^d$ such that the images of any $r$ pairwise disjoint simplices of $K$ do not have a common point.
We show that if $r$ is not a prime power and $d\geq 2r+1$, then there is a counterexample to the topological Tverberg conjecture, i.e., {\it there is an almost $r$-embedding of the $(d+1)(r-1)$-simplex in $\R^d$}.
This improves on previous constructions of counterexamples (for $d\geq 3r$) based on a series of papers
by M. \"Ozaydin, M. Gromov, P. Blagojevi\'{c}, F. Frick, G. Ziegler, and the second and fourth present authors.

The counterexamples are obtained by proving the following algebraic criterion in codimension 2:
{\it If $r\ge3$ and if $K$ is a finite $2(r-1)$-complex then there exists an almost $r$-embedding $K\to \R^{2r}$
if and only if there exists a general position PL map $f \colon K\to \R^{2r}$ such that the algebraic intersection number of the $f$-images of any $r$ pairwise disjoint simplices of $K$ is zero.}
This result can be restated in terms of cohomological obstructions or equivariant maps, and extends an analogous codimension $3$ criterion by the second and fourth authors.
As an application we classify \emph{ornaments} $f\colon S^3 \sqcup S^3 \sqcup S^3\to \R^5$ up to \emph{ornament concordance}.

It follows from work of M. Freedman, V. Krushkal, and P. Teichner that the analogous criterion for $r=2$ is false.
We prove a lemma on singular higher-dimensional Borromean rings, yielding an elementary
proof of the counterexample.
\end{abstract}



\tableofcontents

\section{Main results}

\subsection{The topological Tverberg conjecture and almost $r$-embeddings}\label{s:main}

Throughout this paper, let $K$ be a finite simplicial complex, and let $r$ and $d$ be positive integers.
A map  $f\colon K\to \R^d$ is an \define{almost $r$-embedding} if $f\sigma_1\cap\ldots\cap f\sigma_r=\emptyset$ whenever $\sigma_1,\dots,\sigma_r$ are pairwise disjoint simplices of $K$.
(We stress that this definition depends on the simplicial complex, i.e., a specified triangulation of the underlying polyhedron.)

The well-known \emph{topological Tverberg conjecture}, raised by Bajmoczy and B\'ar\'any~\cite{BB} and Tverberg~\cite[Problem~84]{GS} asserts that
the $(d+1)(r-1)$-dimensional simplex does not admit an almost $r$-embedding in $\R^d$.
This was proved in the case where $r$ is a prime \cite{BB, BShSz} or a prime power \cite{Oz, vo96}, but the case of arbitrary $r$ remained open and was considered a central unsolved problem of topological combinatorics.

Recently and somewhat unexpectedly, it turned out that for $r$ not a prime power and $d\ge3r$ there are counterexamples to the topological Tverberg conjecture.
The construction of these counterexamples follows an approach proposed in \cite{MW14}, which is based on
\begin{itemize}
\item  a general algebraic criterion for the existence of almost $r$-embeddings in \emph{codimension $\ge3$} \cite{MW14, MW} (the \emph{deleted product criterion},
cf. Theorem \ref{t:z-alm3} and Proposition \ref{cor:equiv-alm} below), and

\item a result of \"Ozaydin \cite{Oz} that guarantees that the hypothesis of this criterion is satisfied
whenever $r$ is not a prime power (see \cite[\S3.3 Proof of the \"Ozaydin Theorem 3.5: localization
modulo a prime]{Sk16} for a suitable reformulation and simplified exposition of \"Ozaydin's theorem).
\end{itemize}
There seemed to be a serious obstacle to completing this approach: maps from the
\linebreak
$(d+1)(r-1)$-simplex to $\R^d$ do not satisfy the codimension 3 restriction.
(In a sense, the problem is rather a codimension zero problem.)
Frick \cite{Fr} was the first to realize that this obstacle can be overcome by a beautiful combinatorial trick (Constraint Lemma \ref{p:redu}) discovered  by Gromov \cite{Gr} and independently by Blagojevi\'{c}--Frick--Ziegler \cite{BFZ14}, and that thus the results of \cite{Oz}, \cite{Gr, BFZ14} and \cite{MW} combined yield counterexamples to the topological Tverberg conjecture for $d\geq 3r+1$ whenever $r$ is not a prime power, cf. \cite{Fr, BFZ}.
Using a more involved method (`\emph{prismatic maps}') to overcome the obstacle,
the dimension for the counterexamples was lowered to $d \geq 3r$ in \cite{MW}.
The topological Tverberg conjecture is still open for low dimensions $d<12$, in particular, for $d=2$.


For more detailed accounts of the history of the counterexamples, see the surveys \cite{BBZ},
\cite[\S1 and beginning of \S5]{BZ}, \cite{Sk16}, \cite[\S21.4.5]{Zi17}, \cite{BS} and the references therein.

Here, we improve this and show that counterexamples exist  for $d\ge2r+1$ (see also Remark \ref{r:hist}.a):

\begin{Theorem}\label{t:tve}
There is an almost $6$-embedding of the 70-dimensional simplex in $\R^{13}$.

More generally, if $r$ is not a prime power and $d\ge2r+1$, then there is an almost $r$-embedding of the
$(d+1)(r-1)$-simplex in $\R^d$.
\end{Theorem}

Any sufficiently small perturbation of an almost $r$-embedding is again an almost $r$-embedding.
So the existence of a {\it continuous} almost $r$-embedding is equivalent to the existence of a
{\it piecewise linear (PL)} almost $r$-embedding.

A result closely related to the topological Tverberg conjecture is the following theorem, which generalizes a classical theorem (the case $r=2$) of Van Kampen and Flores \cite{VK}, see also Lemma \ref{l:VK-complex} below.

\begin{Theorem}[$r$-fold van Kampen--Flores Theorem; \cite{Sa91}, {\cite[Corollary in \S1]{Vo96'}}]\label{c:vkfg}
If $r$ is a prime power and $k\ge1$, then there is no almost $r$-embedding of the $k(r-1)$-skeleton of the $(kr+2)(r-1)$-simplex in $\R^{kr}$.
\end{Theorem}

The first ingredient for the proof of Theorem \ref{t:tve} is Part~(a) of the following result, which shows that Theorem~\ref{c:vkfg} fails in a strong sense whenever $r$ is not a prime power.

\begin{Theorem}\label{c:all}
\begin{enumerate}[label=\textup{(\alph*)}]
\item If $k\ge2$ and $r$ is not a prime power, then every finite $k(r-1)$-dimensional complex admits an almost $r$-embedding in $\R^{kr}$.

\item For every fixed $k,r\ge2$, $k+r\ge5$, almost $r$-embeddability of finite $k(r-1)$-dimensional complexes in $\R^{kr}$ is decidable in polynomial time.
\end{enumerate}
\end{Theorem}

For $k\ge3$, Theorem \ref{c:all} is a consequence of \cite{Oz} and \cite{MW}; for $k=2$, it is a result of this paper.
Theorem \ref{c:all} is deduced from Theorem \ref{t:z-alm3} below in \S\ref{s:disj}.

The second ingredient for the proof of Theorem \ref{t:tve} is the following lemma, which was proved in \cite[2.9.c]{Gr}, \cite[Lemma 4.1.iii and 4.2]{BFZ14} (see also \cite[proof of Theorem 4]{Fr}, \cite[proof of Theorem 3.2]{BFZ} and the surveys \cite[Constraint Lemma 3.2]{Sk16}, \cite[\S4, \S5]{BZ}).

\begin{Lemma}[Constraint]\label{p:redu}
If $k,r$ are integers and there is an almost $r$-embedding of the
\linebreak
$k(r-1)$-skeleton of the $(kr+2)(r-1)$-simplex
in $\R^{kr}$, then there is an almost
$r$-embedding of the $(kr+2)(r-1)$-simplex in $\R^{kr+1}$.
\end{Lemma}

Before we proceed, we first show how to derive counterexamples to the topological Tverberg conjecture from these results:

\begin{proof}[Proof of Theorem~\ref{t:tve}]
It is well-known that the general case $d\geq 2r+1$ follows from the `boundary' case $d=2r+1$ \cite[Proposition 2.5]{Lo}, \cite[Lemma 3.1]{Sk16}.
To prove the boundary case, suppose $r$ is not a prime power and let $k=2$. By Theorem \ref{c:all}~(a), there is an almost $r$-embedding of the $2(r-1)$-skeleton of the $(2r+2)(r-1)$-simplex
in $\R^{kr}$. Thus, by Lemma~\ref{p:redu}, there exists an almost
$r$-embedding of the $(2r+2)(r-1)$-simplex in $\R^{kr+1}$.
\end{proof}

The proof of Theorem~\ref{c:all} is based on Theorem~\ref{t:z-alm3} below, which is an extension of a general algebraic criterion for the
existence of almost $r$-embeddings in \emph{codimension $\ge3$} \cite{MW14, MW} to \emph{codimension $2$}.

Assume that $\dim K=k(r-1)$ for some $k\ge1$, $r\ge2$, and that $f\colon K\to \R^{kr}$ is a PL map in general position.
Then preimages $y_1,\ldots,y_r\in K$ of any $r$-fold point $y\in\R^{kr}$ (i.e., of a point having
$r$ preimages) lie in the interiors of $k(r-1)$-dimensional simplices of $K$.
Choose arbitrarily an orientation for each of the $k(r-1)$-simplices.
By general position, $f$ is affine on a neighborhood $U_j$ of $y_j$ for each $j=1,\ldots,r$.
Take a positive basis of $k$ vectors in the oriented normal space to oriented $fU_j$.
\define{The $r$-intersection sign} of $y$ is the sign $\pm 1$ of the basis in $\R^{kr}$ formed by $r$ such $k$-bases.%
\footnote{This is classical for $r=2$ \cite{BE82} and is analogous for $r\ge3$, cf. \cite[\S~2.2]{MW}.}
The {\it algebraic $r$-intersection number}  $f(\sigma_1)\iprod \dots \iprod f(\sigma_r) \in \Z$ is defined as the sum of the $r$-intersection signs of all $r$-fold points $y\in f\sigma_1\cap\ldots\cap f\sigma_r$.
We call a PL map $f$ in general position  a \define{$\Z$-almost $r$-embedding} if
$f\sigma_1\iprod\ldots\iprod f\sigma_r=0$ whenever $\sigma_1,\ldots,\sigma_r$ are pairwise disjoint simplices of $K$.
The sign of the algebraic $r$-intersection number depends on an arbitrary choice of orientations for each $\sigma_i$, but the condition $f\sigma_1\iprod\ldots\iprod f\sigma_r=0$ does not.

\begin{Theorem}\label{t:z-alm3}
If $k\ge2$, $k+r\ge5$ and a finite $k(r-1)$-dimensional simplicial complex is $\Z$-almost $r$-embeddable
in $\R^{kr}$, then it is almost $r$-embeddable in $\R^{kr}$.
\end{Theorem}

The case $r=2$ is a classical result of van Kampen, Shapiro and Wu \cite[Lemma 4.2]{Sk08}.
For $k\ge3$ Theorem \ref{t:z-alm3} is the main result of \cite{MW}. In the present paper, we generalize this to $k\geq 2$.

The proof of Theorem~\ref{t:z-alm3} for $k\ge3$ in \cite{MW} is based on a higher multiplicity generalization \cite[Theorem~17]{MW} of the classical Whitney trick \cite{Wh44} (see, e.g., \cite[Whitney Lemma 5.12]{RS72} for a proof of the Whitney trick in the piecewise-linear setting).
Our proof of Theorem~\ref{t:z-alm3} for $k\geq 2$ is based on a further generalization of the higher-multiplicity Whitney trick that works for $k\geq 2$, namely, the Local and Global Disjunction Theorems \ref{l:ld+3} and \ref{t:elim} that we will formulate in the next subsection (\S\ref{s:disj}). Some readers may consider the resulting proof for $k\geq 2$ simpler than the proof for $k\geq 3$ in \cite{MW}. See also Remarks \ref{rem:Whitney-codim2} and \ref{rem:whitney-vs-disjunction} below for further discussion of the proof ideas and related work.

The analogue of Theorem \ref{t:z-alm3} for $r=2$ and $k=1$ is a classical result of graph theory
(the Hanani-Tutte Theorem \cite{Hanani, Tutte}).
This analogue holds in a stronger form: a mod2-analogue of $\Z$-almost $2$-embeddability in $\R^2$ implies planarity.
For $r=2$ and $k\ge3$ see Remark \ref{r:hist}.b.

The following Theorem~\ref{t:example-2-4} shows that the analogue of Theorem \ref{t:z-alm3} fails for $k=r=2$. Freedman, Krushkal, and Teichner \cite{FKT} proved that there is a 2-complex that admits a $\Z$-almost 2-embedding in $\R^4$, but not an embedding in $\R^4$.
(This implies that the Van Kampen obstruction to embeddability, whose definition we recall
before Proposition \ref{cor:equiv-alm} below, is incomplete for 2-complexes in $\R^4$.)
Here, we strengthen their result and show that their 2-complex does not even admit an almost 2-embedding in $\R^4$.

\begin{Theorem}\label{t:example-2-4}
There exists a finite $2$-dimensional complex that admits a $\Z$-almost $2$-embedding in $\R^4$ but does not admit an almost $2$-embedding in $\R^4$.
\end{Theorem}

Theorem \ref{t:example-2-4} is deduced from the Singular Borromean Rings Lemma \ref{l:bor} below in \S\ref{s:FKT-example}.
This deduction is essentially known \cite{FKT}, \cite[\S7]{Sk08}.

To conclude this subsection,  we state a reformulation of $\Z$-almost $r$-embeddability in $\R^{kr}$, which allows one to deduce Theorem  \ref{c:all} from Theorem \ref{t:z-alm3}.

Let {\it the simplicial $r$-fold deleted product} $K^{\times r}_\Delta$ of $K$ be
\[
K^{\times r}_{\Delta} :=
\bigcup \{ \sigma_1 \times \cdots \times \sigma_r
\; |\; \sigma_i \textrm{ a simplex of }K,
\sigma_i \cap \sigma_j = \emptyset \mbox{ for all $i \neq j$} \},
\]
on which the symmetric group $\mathfrak{S}_r$ acts by permuting the factors.

The group $\mathfrak{S}_r$ acts on the set of real $d\times r$-matrices by permuting the columns.
Denote by $S^{d(r-1)-1}_{\mathfrak{S}_r}$ the set of such matrices with sum in each row equal to zero, and the sum of squares of the matrix elements equal to 1.
This set is homeomorphic to the sphere of dimension $d(r-1)-1$.
This set is invariant under the action of $\mathfrak{S}_r$.

For any general position PL map $f\colon K\to \R^{kr}$, the {\it generalized van Kampen obstruction}
(which is an element of the equivariant cohomology group $H^{kr(r-1)}_{\mathfrak{S}_r}(K^{\times r}_\Delta;\Z_T)$)
is represented by the {\it intersection cocycle} that assigns to each $kr(r-1)$-cell
$\sigma_1\times\ldots\times \sigma_r$ of $K^{\times r}_\Delta$ the algebraic intersection number $f\sigma_1\iprod\ldots\iprod f\sigma_r$; see \cite[\S4]{MW} for details.
The triviality of the obstruction means that the intersection cocycle is null-cohomologous.

\begin{Proposition}\label{cor:equiv-alm} \cite{MW}
Let $K$ be a finite $k(r-1)$-dimensional simplicial complex.
For $k\ge1$ the following conditions are equivalent:
\begin{enumerate}[label=\textup{(\arabic*)}]
\item
$K$ is $\Z$-almost $r$-embeddable in $\R^{kr}$.
\item
The generalized van Kampen obstruction to $\Z$-almost $r$-embeddability of $K$ in $\R^{kr}$ is zero.
\item
There exists a $\mathfrak{S}_r$-equivariant map $K^{\times r}_{\Delta} \to S^{kr(r-1)-1}_{\mathfrak{S}_r}$.
\end{enumerate}
\end{Proposition}

\begin{proof}
The implication (1) $\Rightarrow$ (2) is trivial.
The implication (2) $\Rightarrow$ (1) is \cite[Corollary~44]{MW}.
The equivalence (2) $\Leftrightarrow$ (3) is proved using equivariant obstruction theory, see \cite[Theorem~40]{MW}, \cite[Proposition 3.6]{Sk16}.
\end{proof}

\begin{Proposition}\label{cor:equiv-almoz} \cite{Oz, MW}
Let $K$ be a finite $k(r-1)$-dimensional simplicial complex.
If $r$ is not a prime power, then every $k(r-1)$-complex admits a $\Z$-almost $r$-embedding in $\R^{kr}$.
\end{Proposition}

\begin{proof}
Since $\dim K=k(r-1)$ we have $\dim K^{\times r}_{\Delta} \leq kr(r-1)$.
Now the proposition follows from the implication (3) $\Rightarrow$ (1) of Proposition \ref{cor:equiv-alm}, and
the following result due to \"Ozaydin \cite{Oz} \cite[the \"Ozaydin' Theorem 3.5]{Sk16}:
{\it If $r$ is not a prime power and $\dim K^{\times r}_{\Delta} \leq d(r-1)$, then there is a $\mathfrak{S}_r$-equivariant map $K^{\times r}_{\Delta}\to S^{d(r-1)-1}_{\mathfrak{S}_r}$}.
\end{proof}

\begin{proof}[Proof of Theorem \ref{c:all}]
Part (a) follows from Theorem \ref{t:z-alm3} together with Proposition \ref{cor:equiv-almoz}.

Part (b) follows because by Theorem \ref{t:z-alm3} (together with its trivial converse)
{\it for each $k\ge2$, $k+r\ge5$, almost $r$-embeddability of a $k(r-1)$-complex $K$ in $\R^{kr}$ is equivalent to
each property of Proposition \ref{cor:equiv-alm}.}
Of these, Property (2) is decidable in polynomial time, see \cite[p.~32, Proof of Corollary~9]{MW}
(this is based on algorithms for solving system of linear equations over the integers \cite{St96}).
\end{proof}

\red{Although we drew much inspiration from \cite{MW, Sk00} (and so from other papers which inspired us earlier), this paper is written in a way that, apart from Propositions~\ref{cor:equiv-alm} and \ref{cor:equiv-almoz} and some minor definitions and propositions from \cite{MW}, is formally independent of \cite{MW,Sk00}}

\subsection{Ideas of the proof of Theorem~\ref{t:z-alm3}: Disjunction Theorems}\label{s:disj}

We first formulate the simpler Local Disjunction Theorem, which we consider interesting in itself and which illuminates in simple terms `the core' of the proof of Theorem~\ref{t:z-alm3}.

Let $B^{d}:=[0,1]^d$ denote the standard PL ball and $S^{d-1}=\partial B^d$ the standard PL sphere.
We need to speak about PL balls of different dimensions and we will use the word `disk' for lower-dimensional
objects and `ball' for higher-dimensional ones in order to clarify the distinction (even though, formally, the disk $D^d$ is the same as the ball $B^d$).
We denote by $\partial M$, respectively $\Int M$, the boundary, respectively the interior, of a manifold $M$.
A map $f\colon M \to B^d$ from a manifold with boundary to a ball is called \define{proper},
if $f^{-1}S^{d-1}=\partial M$.
In this paper we work in the PL category, in particular, all disks, balls and maps are PL.

Denote by
$$D=D_1\sqcup\ldots\sqcup D_r$$
the disjoint union of $r$ disks of dimension $k(r-1)$.

\begin{Theorem}[Local Disjunction]\label{l:ld+3}
If $k\ge2$ and $f:D\to B^{kr}$ is a proper general position PL map such that $fD_1\iprod\ldots\iprod f D_r=0\in\Z$, then there is a proper general position PL map $f':D\to B^{kr}$ such that $f'=f$ on $\partial D$ and
$f'D_1\cap\ldots\cap f' D_r=\emptyset$.
\end{Theorem}

The condition $fD_1\iprod\ldots\iprod f D_r=0$ can be called {\it algebraic triviality}, and the condition
\linebreak
$fD_1\cap\ldots\cap fD_r=\emptyset$ can be called {\it geometric triviality}.

The case $r=2$ of Theorem \ref{l:ld+3} is known, see Remark \ref{r:hist}.c.
The case $k\ge3$ is essentially proved in \cite[Theorem 17]{MW} (in fact, the case $k\ge3$ of Theorem \ref{l:ld+3} is the only part of quite technical \cite[Theorem 17]{MW} required to prove Theorem \ref{t:z-alm3} for $k\ge3$).
The case $r\ge3$, $k=2$ is a result of this paper.

Theorem \ref{l:ld+3} for $r\ge3$ follows from the Global Disjunction Theorem \ref{t:elim}.(a)-(b) below.

The analogue of Theorem~\ref{l:ld+3} for $k=1$ clearly holds when $r=2$ and fails for each $r\ge3$:

\begin{Theorem}\label{c:nld1}
For each $r\ge3$ take $k=1$ in the definition of $D$.
Then there is a proper general position PL map $f:D\to B^r$ such that $fD_1\iprod\ldots\iprod f D_r=0$
but there is no proper general position PL map $f':D\to B^r$ such that $f'=f$ on $\partial D$ and $f'D_1\cap\ldots\cap f'D_r=\emptyset$.
\end{Theorem}

\begin{figure}[h]
\centerline{\includegraphics[width=5cm]{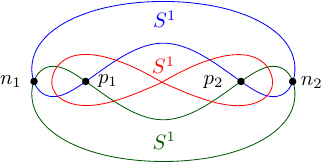} }
\caption{The boundary of an example corresponding to Theorem~\ref{c:nld1} for $r=3$.}
\label{f:r=3}
\end{figure}

As an example corresponding to Theorem~\ref{c:nld1} one can take an extension of the map $f|_{\partial D}$
constructed in the proof of Lemma~\ref{l:cdm1} below.
For $r=3$ see Figure \ref{f:r=3}; this construction might be known.
For $r=3$ Theorem~\ref{c:nld1} could be reproved using Figure \ref{f:r=3} and \cite{Me03}.

The Local Disjunction Theorem \ref{l:ld+3} can be {\it globalized}, i.e. generalized to other connected orientable manifolds instead of disks and balls, including closed manifolds in $\R^d$ rather than proper manifolds in $B^d$.
For $k\ge3$ see \cite[Theorem 17]{MW}, for $k=r=2$ see \cite{ST12} and references therein.
Let us state a {\it polyhedral} global version required  to prove Theorem \ref{t:z-alm3}.
(For {\it ornamental} global versions see \S\ref{s:ornsta}.)

We call a point $y\in\R^d$ a \define{global $r$-fold point} of a map $f\colon K\to \R^d$ if $y$ has
$r$ preimages lying in pairwise disjoint simplices of $K$, i.e., $y\in f\sigma_1\cap\ldots\cap f\sigma_r$
and $\sigma_i\cap\sigma_j=\emptyset$ for $i\ne j$.

(Thus, $f$ is an almost $r$-embedding if and only if it has no global $r$-fold points.)

\smallskip
\noindent\textbf{Assertion $\boldsymbol{(D_{k,r})}$.}
{\it Let
\begin{itemize}

\item $K$ be a finite $k(r-1)$-dimensional simplicial complex,

\item $f\colon K\to B^{kr}$ a general position PL map,

\item $\sigma_1,\ldots,\sigma_r$ pairwise disjoint simplices of $K$,

\item $x,y\in f\sigma_1\cap\ldots\cap f\sigma_r\subset\Int B^{kr}$ two global $r$-fold points of opposite $r$-intersections signs.
\end{itemize}
\noindent
Then there is a general position PL map $f'\colon K\to B^{kr}$ such that $f=f'$ on $K-(\Int\sigma_1\sqcup\dots\sqcup\Int\sigma_r)$, and $f'$ has the same global $r$-fold points with the same signs as $f$ except that $x,y$ are not global $r$-fold points of $f'$.}

\medskip
This can be informally described as  `cancelation of a pair of global $r$-fold points of opposite sign'.
The Local Disjunction Theorem~\ref{l:ld+3} is such a cancelation {\it in a restricted local situation}.
So these are partial analogues of the Whitney trick, but we prefer a self-descriptive name.

Assertion $(D_{1,2})$ is a version of `redrawing of a graph in the plane' \cite[\S4]{Sc}.
It would be interesting to know if it is true.

Theorem~\ref{t:z-alm3} ($\Z$-almost $r$-embeddability implies almost $r$-embeddability),
as well as Theorem \ref{thm_correspondance3} below (classification of ornaments) follow from the following Global Disjunction Theorems~\ref{t:elim}.(a)-(b).

\begin{Theorem}[Global Disjunction]\label{t:elim}
\begin{enumerate}[label=\textup{(\alph*)}]

\item \cite{MW} Assertion $(D_{k,r})$  is true for each $k\ge3$ and $r\ge2$.

\item Assertion $(D_{2,r})$  is true for each $r\ge3$.

\item Assertion $(D_{2,2})$  is false.

\item Assertion $(D_{1,r})$  is false for each $r\ge3$.
\end{enumerate}
\end{Theorem}

\begin{proof}[Proof of Theorem~\ref{t:z-alm3} assuming the Global Disjunction Theorems \ref{t:elim}.(a)-(b)]
Let $f\colon K\to \R^{kr}$ be a $\Z$-almost $r$-embedding.
Take pairwise disjoint simplices $\sigma_1,\dots,\sigma_r$ of $K$ with
$f\sigma_1\cap\ldots\cap f\sigma_r\neq \emptyset$.
Since $f$ is a $\Z$-almost $r$-embedding, $f\sigma_1\cap\ldots\cap f\sigma_r$ consists of pairs of global $r$-fold points of opposite sign.
By assertion $(D_{k,r})$, we eliminate these pairs one by one, without introducing any new global $r$-fold points in the process.
By repeating this for every $r$-tuple of pairwise disjoint simplices, we obtain an almost $r$-embedding
$K\to \R^{kr}$.
\end{proof}

The Global Disjunction Theorems \ref{t:elim}.(a)-(b) are proved in \S\ref{s:proasmod}.
The Global Disjunction Theorem \ref{t:elim}.c follows because assertion $(D_{2,2})$ implies the negation of
Theorem \ref{t:example-2-4} analogously to the above proof.
The Global Disjunction Theorem \ref{t:elim}.d follows because assertion $(D_{1,r})$ implies the negation of
Theorem \ref{c:nld1} analogously to the above proof.

\begin{Remark}\label{rem:Whitney-codim2}
It is well-known that the Whitney trick works in codimension $\ge3$ and fails in codimension 2 \cite{KM} without
an assumption of simple connectivity \cite[Whitney~Lemma~5.12.2 and p.~72, the first condition (2)]{RS72}, which is not satisfied in our applications.

Usually it is non-trivial to make `Whitney-trick-arguments' work for codimension $2$; a famous example is Freedman's proof of the Poincar\'e conjecture in dimension 4 \cite[Chapter 13]{Ki}. The non-triviality of Theorem \ref{t:z-alm3} for $k=2$ is also seen from Theorem \ref{t:example-2-4} (which shows that the analogous result for $r=2$ is false) and from Theorems \ref{l:ld+3} and \ref{thm_correspondance3}.a (which show that the analogous result for {\it ornaments} is true even for $r=2$). In other words, the codimension 2 situation is sufficiently delicate to provide different results for different $r$ and for different conditions on $r$-fold intersections.

A crucial insight for making a version of the Whitney trick work in our context is that, unlike in the classical case of embeddings, we can permit singularities (as long as they are of multiplicity less than $r$). This allows us to make modifications by \emph{homotopy} as opposed to isotopy, which gives us more flexibiliy. Together with a restructuring of the arguments, this also leads to a simpler proof of the codimension 3 result, which is presented here and in \cite[\S3.5
]{Sk16}, and which some readers may wish to read before studying the proof for codimension $k\geq 2$ in \S\ref{s:proasmod}. For further comments on the proof ideas and related work, see also Remark~\ref{rem:whitney-vs-disjunction} below.
\end{Remark}

\subsection{Classification of ornaments and doodles}\label{s:ornsta}

In this subsection we describe another application of our methods in the topological context of higher multiplicity linking.

Throughout this subsection $S=S_1\sqcup\ldots\sqcup S_r$ will denote a disjoint union of $r$ copies of $S^n$ and
$D=D_1\sqcup\ldots\sqcup D_r$ a disjoint union of $r$ copies of $D^{n+1}$; the dimensions of $S,D$ will be clear from the context.

An \define{$r$-component $n$-ornament in $S^d$} is a general position PL map $f:S\to S^d$
such that $fS_1\cap\ldots\cap fS_r=\emptyset$.%

Let $r\ge2$ and $f$ be an $r$-component $(k(r-1)-1)$-ornament in $S^{kr-1}$.
Extend $f$ to a general position PL map $g:D\to B^{kr}$ (the extension is constructed e.g., by `coning'
each $f|_{S_i}$ to interior point of $B^{kr}$, a distinct cone point for each component).
Define \define{the $r$-linking number} of $f$ by
$$\lk f := g D_1\iprod \ldots \iprod g D_r\in\Z.$$
This definition is a natural generalization of the classical linking number (obtained for $r=2$),
and $\mu$-invariant of \cite{FT} (obtained for $r=3$ and $k=1$).%
\footnote{It is also similar in spirit, but different from, the {\it Massey-Milnor triple linking number}
\cite{Po}, \cite[\S4.5 `Massey-Milnor number modulo 2']{Sk},  which distinguishes Borromean rings from the standard link.
The 3-linking number of Borromean rings is not defined, because they do not form
an $r$-component $(k(r-1)-1)$-dimensional ornament in $S^{kr-1}$ for any $k,r$.
For the relation see \cite[Theorem 3]{FT}.}
Analogously to the case $r=2$ one can check that $\lk f$ is well-defined, i.e., is independent of the choices of the extension $g$.%
\footnote{By induction, it suffices to prove this for two extensions $g$ and $g'$ that agree on all but one disk, say $g|_{D_i}=g'|_{D_i}$ for every $i<r$.
Then $gD_r\cup(-g'D_r)$ carries an integer cycle in $B^{kr}$.
This cycle is the boundary of some integral $(k(r-1)+1)$-dimensional chain $C$ in $B^{kr}$ with
$C\cap S^{kr-1}=fS_r$.
Since $f S_{r-1}=-\partial g D_{r-1} \subset S^{kr-1}$, $g D_i \cap S^{kr-1}=f S_i$, and $f$ is an ornament, by \cite[Definition~29 and Lemma~28]{MW} it follows that
$g D_1\iprod \ldots \iprod g D_r - g' D_1\iprod \ldots \iprod g' D_r
= g D_1 \iprod g D_{r-2} \iprod f S_{r-1} \iprod C = 0$.}

Clearly, if an $r$-component $(k(r-1)-1)$-ornament in $S^{kr-1}$ bounds a map $g:D\to B^{kr}$ such that $gD_1\cap\ldots\cap gD_r=\emptyset$, then the ornament has zero $r$-linking number.
The converse is true for every $k\ge2$, which is a generalization of the Local Disjunction Theorem \ref{l:ld+3} and a particular case of Theorem~\ref{thm_correspondance3}.a below.
For $k=1$, the converse clearly holds when $r=2$ and fails for each $r\ge3$ by Theorem \ref{c:nld1}.

Denote $I := [0,1]$.
An \define{ornament concordance} is a map $F: S\times I \rightarrow S^d \times I$ such that
$$F(\cdot,t) \subset S^d \times \{ t \}\quad\text{for each}\quad t=0,1,
\quad\text{and}\quad F(S_1\times I)\cap F(S_2\times I)\cap\ldots\cap F(S_r\times I) = \emptyset.$$
Analogously to the case $r=2$ \cite[\S77]{SeTh80} $\lk f$ is invariant under ornament concordance.

An ornament is called a \define{doodle} if its restriction to each connected component is an embedding.
Likewise, a \define{doodle concordance} is an ornament concordance such that its restriction to each connected component is an embedding.

An ornament [doodle] is called \textbf{trivial} if it is concordant to an ornament [doodle] whose components lie in pairwise disjoint balls.
For $(r-1)d>rn+1$ ($\Leftrightarrow(r-1)(d+1)>r(n+1)$) any $r$-component $n$-ornament in $S^d$ is trivial by general position.

\begin{Theorem}\label{thm_correspondance3}
The $r$-linking number defines a 1-1 correspondence between $\Z$ and the set of
\begin{enumerate}[label=\textup{(\alph*)}]
\item ornament concordance classes of $r$-component $(k(r-1)-1)$-ornaments (or doodles) in $S^{kr-1}$ for each $r,k\ge2$.
\item doodle concordance classes of $r$-component $(k(r-1)-1)$-doodles in $S^{kr-1}$ for each $r\ge2$, $k\ge3$.
\end{enumerate}
\end{Theorem}

Theorem~\ref{thm_correspondance3} for $r=2$ is well-known.
The case $k\ge3=r$ of Theorem~\ref{thm_correspondance3}.a is due to Melikhov~\cite[p. 7]{Me}.
For $k\geq3$ and each $r$, Theorem~\ref{thm_correspondance3} can be derived from \cite[Theorem 17]{MW}, for Part~(b) using  Remark~\ref{rem:whitney-vs-disjunction}.b. Theorem~\ref{thm_correspondance3}.a for $r\ge3$, $k=2$ is a result of this paper.
Our proof (\S\ref{s:ornpro}) works for any $r,k\ge2$.

The analogue of Theorem \ref{thm_correspondance3}.a for $k=1$ and $r=2$ is clearly true,
for $k=1$ and each $r\ge3$ it is false by Theorem \ref{c:nld1}. See Remark~\ref{rem:ornaments} below for further comments on ornaments.
\medskip

The Local Disjunction Theorem~\ref{l:ld+3} is a particular case of the following `ornamental' analogue
of Theorem \ref{t:z-alm3}.
The existence of an ornament is trivial, so we state a non-trivial {\it relative} version.

\begin{Theorem}\label{t:genorn}
Assume that  $k,r\ge2$,
\begin{itemize}

\item $K=K_1\sqcup\ldots\sqcup K_r$ is a finite simplicial $(k-1)r$-complex,

\item $f:K\to B^{kr}$ is a general position map,

\item $L:=f^{-1}S^{kr-1}\subset K$ is a subcomplex and $f|_L$ is an $r$-component ornament in $S^{kr-1}$,

\item $f\sigma_1\iprod\ldots\iprod f\sigma_r=0\in \Z$
whenever $\sigma_1\subset K_1,\ldots,\sigma_r\subset K_r$ are $(k-1)r$-simplices of $K$.
\end{itemize}
Then there is an $r$-component ornament $f':K\to B^{kr}$ such that $f'=f$ on $L$.
\end{Theorem}

For $r=2$ this is known \cite{Sk00}.
For $r\ge3$ this follows from the Global Disjunction Theorems \ref{t:elim}.(a)-(b) analogously to the above proof of the injectivity in Theorem~\ref{thm_correspondance3}.a.
A proof which works for any $k,r\ge2$ could perhaps be given by stating and proving the ornamental version of
 the Global Disjunction Theorems \ref{t:elim}.(a)-(b) (which works even for $k=r=2$).
It is interesting to compare the case $k=r=2$ of Theorem \ref{t:genorn} to the Global Disjunction Theorems \ref{t:elim}.(c).

\section{Proofs}\label{s:proofs}

\subsection{Proof of the Global Disjunction Theorems \ref{t:elim}.(a)-(b)}\label{s:proasmod}

Informally speaking, the first step in the proof of the Global Disjunction Theorems \ref{t:elim}.(a)-(b) is to make the $(r-1)$-intersection $f\sigma_1\cap\ldots\cap f\sigma_{r-1}$ connected.
See the following Lemmas \ref{l:surg} and \ref{t:gdi}.

Throughout this section, let us fix orientations on balls $B^d$ and disks $D^m$.

\begin{Lemma}[Surgery of Intersection]\label{l:surg}
Assume that $d-2\ge p,q$ and that $f:D^p\to B^d$, $g:D^q\to B^d$ are proper embeddings in general position such
that $fD^p\cap gD^q$ is a proper submanifold (possibly disconnected) of $B^d$ containing points $x,y$.
\begin{enumerate}[label=\textup{(\alph*)}]
\item If $p+q>d$ then there is a proper general position map $f':D^p\to B^d$
with the following properties:
\begin{itemize}
\item $f'=f$ on $\partial D^p$ and on a neighborhood of $\{f^{-1}x,f^{-1}y\}$;
\item $x,y$ lie in the interior of an embedded $(p+q-d)$-disk  contained in $f'D^p\cap gD^q$.
\end{itemize}
\item If $p+q=d\ge q+3$, \ $\{x,y\}=fD^p\cap gD^q$ and $x,y$ have opposite $2$-intersection sign, then there is a general position map $f':D^p\to B^d$ such that $f'=f$ on $\partial D^p$ and $f'D^p\cap gD^q=\emptyset$.
\end{enumerate}
\end{Lemma}

This lemma is known for $d-3\ge p,q$ (then Part~(b) is the classical Whitney trick, and for Part (a) see Remark~\ref{rem:whitney-vs-disjunction}.a), and Part (b) is known also for $q=2$ \cite[Lemma 2.4]{Sp}.
Passage to $d-2\ge p,q$ in (a), or to $d-2=p$ in (b), requires losing the injectivity properties of $f,g$.
Part~(b) $d-2=p\ge3$ is proved by seeing that $f|_{\partial D^p}$ is null-homotopic in $B^d-gD^q$
(an analogue for $d-2=p=2$ is discussed in Remark \ref{r:hist}.c).

In what follows, we first use the Surgery of Intersection Lemma \ref{l:surg} to prove the following Lemma~\ref{t:gdi} and the Global Disjunction Theorem \ref{t:elim}.(a)-(b).
The proof of the Surgery of Intersection Lemma \ref{l:surg} is then given at the end of this subsection.

In the rest of this section, we abbreviate $B^{kr}$ to $B$.

\begin{Lemma}\label{t:gdi}
Assume that $k,r\ge2$,
\begin{itemize}
\item $K$ is a finite $k(r-1)$-dimensional simplicial complex,

\item $f\colon K\to B$ a general position PL map,

\item $\sigma_1,\ldots,\sigma_r$ are pairwise disjoint simplices of $K$,

\item  $x,y\in f\sigma_1\cap\ldots\cap f\sigma_r\subset\Int B$ are two global $r$-fold points of opposite $r$-intersections signs.
\end{itemize}
Then for each $n=1,\ldots,r-1$ there is a general position PL map $f'\colon K\to B$ such that
\begin{itemize}
\item $f=f'$ on $K-(\Int\sigma_1\sqcup\dots\sqcup\Int\sigma_r)$,

\item  $x,y$ lie in the interior of an embedded $k(r-n)$-disk  contained in $f'\sigma_1\cap\ldots\cap f'\sigma_n$, and

\item  $f'$ has the same global $r$-fold points with the same signs as $f$.
\end{itemize}
\end{Lemma}

\begin{proof}
The proof is by induction on $n$.
The base $n=1$ follows by setting $f'=f$.
The required disk is then a small regular neighborhood in $f\sigma_1$ of a path in $f\sigma_1$ joining $x$ to $y$
and avoiding the self-intersection set $\{x\in K\ :\ | f^{-1}fx| \geq 2\}$ of $f$.

In order to prove the inductive step assume that $n\ge2$ and the points $x,y$ lie in the interior of an embedded $k(r-n+1)$-disk $\sigma_-\subset f\sigma_1\cap\ldots\cap f\sigma_{n-1}$.
By general position
$$\dim(\sigma_-\cap f\sigma_n)\le k(r-n+1)+k(r-1)-kr=k(r-n).$$
Since $f$-preimages of $x$ lie in the interiors of $\sigma_1,\ldots,\sigma_r$, the intersections of $f\sigma_i$ and small regular neighborhoods of $x,y$ in $B$ equal to the intersections of affine spaces and the neighborhoods.
Hence the regular neighborhoods of $x,y$ in $\sigma_-\cap f\sigma_n$ are $k(r-n)$-balls.

Take points $x',y'$ in such balls.
Take general position paths
$\lambda_+\subset f\Int\sigma_n\quad\text{and}\quad \lambda_-\subset \sigma_-$
joining $x'$ to $y'$.
By general position dimension of the self-intersection set of $f$  does not exceed $2k(r-1)-kr<k(r-1)-1$.
So the union $\lambda_+\cup \lambda_-$ is an embedded circle in $\Int B$.
Since $k,r\ge2$, we have $kr\ge4$.
Hence by general position this circle bounds an embedded $2$-disk $\delta\subset\Int B$.
Since $k\ge2$, we have $k(r-1)+2\le kr$.
Hence by general position
$$\delta\cap fK=\lambda_+\cup\lambda_-\sqcup\{fp_1,\ldots,fp_s\}$$
for some points $p_1,\ldots,p_s\in K$ outside the self-intersection set of $f$ and the
$(k(r-1)-1)$-skeleton of $K$, and $s=0$ for $k\ge3$.

Let $O\delta$ be a small regular neighborhood of $\delta$ in $\Int B$.
Then $O\delta$ is a $kr$-ball and
$f^{-1}O\delta$ is the union of

\begin{itemize}
\item a regular neighborhood $D_n\cong D^{k(r-1)}$ of the arc $f|_{\sigma_n}^{-1}\lambda_+$ in $\sigma_n$;


\item regular neighborhoods $D_i\cong D^{k(r-1)}$ of the arcs $f|_{\sigma_i}^{-1}\lambda_-$ in $\sigma_i$
for each $i=1,\ldots,n-1$;

\item pairwise disjoint $k(r-1)$-disks that are regular neighborhoods of $p_j$ in the $k(r-1)$-simplices
of $K$ containing them, for each $j=1,\ldots,s$; these disks are disjoint from the self-intersection set of $f$,
and their $f$-images are disjoint from $fD_1\cup\ldots\cup fD_n$.
\end{itemize}


Then $f|_{D_i}:D_i\to O\delta$ is proper for each $i=1,\ldots,n$, and $\sigma_-\cap O\delta$ is a proper $k(r-n+1)$-ball in $O\delta$.
Since the regular neighborhoods of $x,y$ in $\sigma_-\cap f\sigma_n$ are $k(r-n)$-balls,
the set $\sigma_-\cap O\delta\cap fD_n$ is a proper $k(r-n)$-submanifold of $O\delta$.
Hence we can apply the Surgery of Intersection Lemma \ref{l:surg}.a to $fD_n$ and $\sigma_-\cap O\delta$ in $O\delta$.
For the obtained map $f':D_n\to O\delta$ the points
$x,y\in f'\sigma_1\cap\ldots\cap f'\sigma_r\subset\Int B$ are two global $r$-fold points of opposite $r$-intersections signs, lying in the interior of an embedded $k(r-n)$-disk  contained in $\sigma_-\cap f'D_n$.
Extend $f'$ by $f$ outside $D_n$.

Clearly, the first two bullet points in the conclusion of Lemma \ref{t:gdi} are fulfilled.
All the global $r$-fold points of $f$ lie outside $O\delta$, and $f=f'$ outside $O\delta$.
Therefore all the global $r$-fold points of $f$ are also global $r$-fold points of $f'$, and they have the same sign.
It remains to check that $f'$ does not have new global $r$-fold points inside $O\delta$.
In $O\delta$ the map $f'$ can have global points of multiplicity at most $n$ in $f'D_n\cap fD_1\cap\ldots\cap fD_{n-1}$, or of multiplicity $2$ in the intersection of $f'D_n$ with the $f$-image of a small neighborhood of some $p_j$.
Since $r>n\geq 2$, none of these global points are $r$-fold.

Thus the map $f'$ is as required.
\end{proof}

\begin{proof}[Proof of the Global Disjunction Theorems \ref{t:elim}.(a)-(b)]
By Lemma \ref{t:gdi} for $n=r-1$ we may assume that the points $x,y$ lie in the interior of an embedded $k$-disk $\sigma_-\subset f\sigma_1\cap\ldots\cap f\sigma_{r-1}$.
Choose orientations of $\sigma_1,\ldots,\sigma_{r-1}$.
These orientations define an orientation on $\sigma_-$ (this is analogous to the definition of the $r$-intersection sign given before Theorem~\ref{t:z-alm3}, cf. \cite[\S2.2]{MW} for a longer formal exposition).
Since $x,y\in f\sigma_1\cap\ldots\cap f\sigma_r$ have opposite $r$-intersections signs,
$x,y\in \sigma_-\cap f\sigma_r$ have opposite $2$-intersections signs \cite[Lemma 27.cd]{MW}.

Analogously to the proof of Lemma \ref{t:gdi} (except that we start from $x,y$ not from $x',y'$) we construct a $kr$-ball $O\delta\subset\Int B$ and $k(r-1)$-disks $D_i\subset\Int\sigma_i$ for $i=1,\ldots,r$,
such that $x,y\in O\delta$ are the only global $r$-fold points in $O\delta$ and $f|_{D_i}:D_i\to O\delta$ is proper.

Since either $r\ge3$ or $k\ge3$, we have $kr\ge\dim\sigma_-+3$.
So we can apply the Surgery of Intersection Lemma \ref{l:surg}.b to $fD_r$ and $\sigma_-\cap O\delta$ in $O\delta$.
For the obtained map $f':D_r\to O\delta$ we have $\sigma_-\cap f'D_r=\emptyset$.
Extend $f'$ by $f$ outside $D_r$.

Clearly, $f=f'$ on $K-(\Int\sigma_1\sqcup\ldots\sqcup\Int\sigma_r)$.
Since $f=f'$ outside of $D_r$, all the global $r$-fold points of $f$ except $x,y$ are also global $r$-fold points of $f'$, and they have the same sign.
It remains to check that $f'D_r$ contains no global $r$-fold points of $f'$.
Recall the description of $f^{-1}O\delta$ from the bullet points in the proof of Lemma \ref{t:gdi}.

If $r = 2$, then $k\geq 3$, so $s = 0$.
Also $O\delta\cap\sigma_-=f(D_1)=f'(D_1)$.
So $f'(K)\cap O\delta = f'(D_1)\sqcup f'(D_2)=(O\delta\cap\sigma_-)\sqcup f'(D_2)$, where the union is disjoint by the construction of $f'$.
Therefore $f'D_2=f'D_r$ contains no global $2$-fold points of $f'$.

If $r>2$, then $f'|_{K\setminus D_r}$ has no $(r-1)$-fold point in $O\delta$ except for $\sigma_-\cap O\delta$.
By construction $\sigma_-\cap f'D_r=\emptyset$, so again $f'D_r$ contains no global $r$-fold points of $f'$.
\end{proof}

\begin{proof}[Proof of the Surgery of Intersection Lemma \ref{l:surg}.a]
To simplify notation, let us write
$$Q:=gD^q \quad\text{and} \quad M:=fD^p\cap Q$$
throughout this proof.
Furthermore, let $m:=\dim M=p+q-d$.
Note that the assumptions on the dimensions $p,q,d$ imply that $m+2\leq p,q$ and $d\geq 5$.

The chosen orientations of $B^d,D^p$, and $D^q$ define an orientation on $M$ (this is analogous to the definition of the $r$-intersection sign given before Theorem~\ref{t:z-alm3}, cf. \cite[\S2.2]{MW} for a longer formal exposition).

Let us first assume that $x$ and $y$ lie in different connected components of $M$.
We proceed in two steps to reduce this case to the case where $x$ and $y$ lie in the same connected component of the intersection;
it will then be easy to deal with the latter situation.
\medskip

\emph{Step 1. Ambient 1-surgery.
\textup{(}``piping''.\textup{)}}
Pick two generic points $a ,b \in M$ such that $a$ lies in the
same connected component of $M$ as $x$, and $b$ lies in the same connected component of $M$ as $y$.
Pick a general position path $\ell \subset Q$ connecting $a$ and $b$.

\begin{figure}[h]
\centerline{\includegraphics[width=16cm]{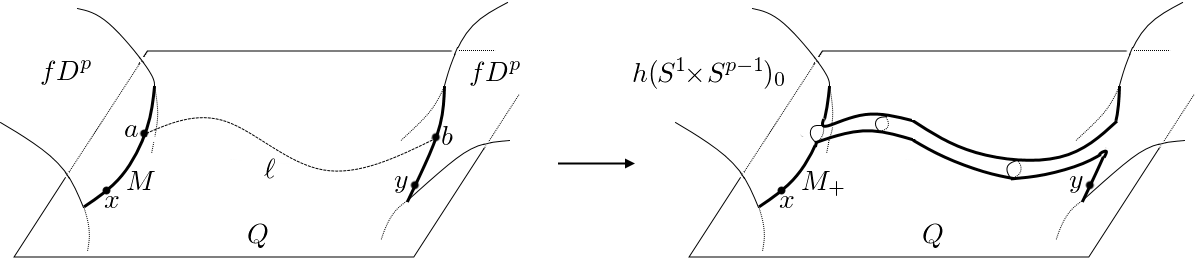} }
\caption{Piping}
\label{f:piping}
\end{figure}

By general position, $\ell$ is disjoint (and hence at a positive distance from) the set of points at
which $Q$ is not locally flat in $B^d$ (see \cite[p.~50]{RS72}
for the definition of local flatness); this follows because {\it the set of non-locally flatness points of the codimension $\geq 2$ submanifold $Q\subset B^d$ has codimension $\ge 2$ in $Q$.}%
\footnote{Let us prove the latter statement.
Take a triangulation of $B^d$ such that $Q$ is a subcomplex of this triangulation.
For each point $c\in Q$ the pair $(\lk_{B^d}c,\lk_Qc)$ is a codimension $\geq 2$ pair of spheres.
If $c$ is outside the codimension 2 skeleton of this subcomplex, then $\dim\lk_Qc\in\{-1,0\}$.
Hence the pair $(\lk_{B^d}c,\lk_Qc)$ is unknotted.
Thus, $Q$ is locally flat in $B^d$ at $c$.}

We now perform ambient 1-surgery on $M$ in $Q$
as described in \cite[pp.~67-68]{RS72}
(where this procedure is called ``piping'') to obtain connected manifold $M_+$; more precisely, take an embedding
$L:I\times D^{m}\to Q$ that satisfies the following properties (where we use $m\leq q-2$ for the second property):
\begin{itemize}
\item $L(I\times0)=\ell$,
\item $M \cap L(I\times D^{m}) = L(\{0,1\} \times D^{m})$ is a regular neighborhood of $\{a,b\}$ in $M$,
\item the orientation of $M$ on this neighborhood is compatible with the `boundary' orientation of
$L(\{0,1\} \times D^{p})$, and
\item $L(I\times D^{m})$ is disjoint from $x,y$ and from any non-locally flatness points of $Q$ in $B^d$.
\end{itemize}
We define
$$M_+:= \big(M \setminus L(\{0,1\} \times \Int{D}^{m})\big) \cup L(I\times \partial D^{m}).$$
By construction, $x$ and $y$ lie in the same component of $M_{+}$, and $M_+$ is orientable.
We give $M_+$ the orientation induced by that of $M$.

By general position $fD^p$ and $Q$ are transverse at $\{a,b\}$.
Since $\ell$ does not contain non-locally flatness points of $Q$ in $B^d$,
the submanifold $Q$ is locally flat in $B^d$ in a neighborhood of $\ell$.
Hence, we can extend $L$ to an embedding $L:I\times D^{p}\to B^d$ such that
\begin{itemize}
\item $Q\cap L(I\times D^{p})=L(I\times D^m)$,
\item $fD^p \cap L(I\times D^{p}) = L(\{0,1\} \times D^{p})$ is a regular neighborhood of $\{a,b\}$ in $fD^p$, and
\item the orientation of $fD^p$ on this neighborhood is compatible with the `boundary' orientation of
$f|_{D^p}^{-1}(L(\{0,1\} \times D^{p}))$.
\end{itemize}
Denote by $(S^1\times S^{p-1})_0$ the manifold $S^1\times S^{p-1}$ with an open $p$-disk removed.
Let
$$h\colon (S^1\times S^{p-1})_0 \to B$$
be the proper embedding obtained by adding the embedded $1$-handle $L(I\times \partial D^{p})$
to $fD^p$; thus,
$$\im h = \big(fD^p \setminus L(\{0,1\} \times \Int{D}^{p})\big) \cup L(I\times \partial D^{p}).$$
By construction, the intersection $\im h \cap Q=M_+$ is connected.
(Note that $h$ is an embedding of $(S^1\times S^{p-1})_0$, not of $D^p$; this will be repaired in the next step.)

\medskip
\emph{Step 2. Ambient 2-surgery.
\textup{(}``unpiping''\textup{)}.} We now perform ambient 2-surgery on $\im h$ in $Q$
to obtain a proper embedding $f'\colon D^p\to B$ such that $\im h \cap Q=M_+ \subseteq f'D^p \cap Q$
(this is analogous to \cite[Lemma~38]{MW}, where the corresponding operation is called  ``unpiping'').

\begin{figure}[h]
\centerline{\includegraphics[width=12cm]{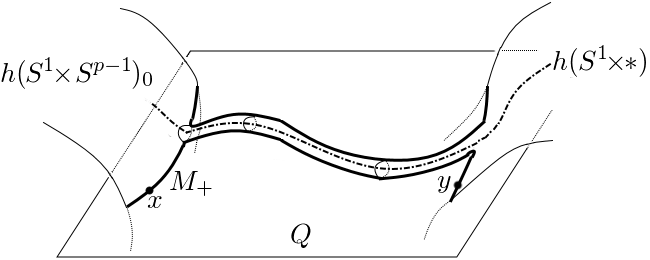} }
\caption{Unpiping}
\label{f:unpiping}
\end{figure}

Pick a point $\ast \in S^{p-1}$ in general position with respect to $h$,
and a general position embedded $2$-disk $\delta \subset B^d$ such that $\partial\delta =h(S^1\times \ast)$.
By general position, $\partial \delta$ is disjoint from $Q$, and $\delta$ intersects $Q$ in a finite set (empty if $q\leq d-3$) of points disjoint from $\im h$.

Denote by $O\delta$ a small regular neighborhood of $\delta$ in $B$.
Take a small regular neighborhood $U\cong D^{p-1}$ of $\ast$ in $S^{p-1}$.
We may assume that $h(S^1\times U)\subset O\delta$. Since $O\delta \cong B^{d}$, the restriction
$S^1\times\partial U\rightarrow O\delta$ of $h$ extends to a map $j\colon D^2\times\partial U\rightarrow O\delta$.

Let
$$\Delta:=\big((S^1\times S^{p-1})_0 \setminus (S^1\times \Int U)\big) \cup (D^2 \times \partial U) \cong D^p.$$
Define
$$
f'\colon \Delta \to B^d \quad\text{by}\quad f'(x):=
\begin{cases}
h(x) &\text{if}\quad x\in N \setminus (S^1\times \Int U),\\
j(x) &\text{if}\quad x\in D^2\times \partial U.
 \end{cases}.
$$
By construction of $f'$, $f=f'$ on $\partial D^p$ and in a neighborhood of $x,y$ (identifying $\Delta$ with $D^p$).
Moreover, $f'\Delta \cap Q$ consists of the manifold $\im h \cap Q$ plus possibly some additional further components.
In particular, $x$ and $y$ lie in the same connected component $f'\Delta \cap Q$, and this component is a manifold of dimension $m=p+q-d$.
\medskip

To complete the proof, take a general position path $\ell \subset f'\Delta \cap Q$ connecting $x$ and $y$.
Then a regular neighborhood of $\ell$ in $f'\Delta \cap Q$ is then an $m$-disk that contains $x$ and $y$.
\end{proof}

\begin{proof}[Proof of the Surgery of Intersection Lemma \ref{l:surg}.b]
Denote $X:=B^d-gD^q$.
Consider the composition
$$\pi_{p-1}(X)\overset{h}\to H_{p-1}(X)\overset{\cong}\to H_0(D^q) \cong\Z$$
of the Hurewicz homomorphism and the (homological) Alexander duality
isomorphism.
This composition carries $[f|_{S^{p-1}}]$ to $fD^p\iprod gD^q$.
\footnote{This is one of the equivalent definitions of Alexander duality isomorphism, see
Alexander Duality Lemmas of \cite{Sk08', Sk10}.
Another equivalent definition is as follows.
Take a small oriented disk $D^p_g\subset B^d$ whose intersection with $gD^q$ consists of exactly one point
of sign $+1$ and such that $\partial D^p_g\subset X$.
Then the Alexander duality carries the generator $\partial D^p_f$ of
$H_{p-1}(X)$ to the generator of $H_0(D^q)$ defined by the orientation of $D^q$.}
The assumptions of part (b) imply that $fD^p\iprod gD^q=0$.
By general position $X$ is $(p-2)$-connected.
Since $p\ge3$, we have $p-2\ge1$, so by the Hurewicz theorem $h$ is an isomorphism.
Hence the restriction $f:S^{p-1}\to X$ is null-homotopic.
Thus there is an extension $f':D^p\to X$ of the restriction.
This is the required map.
\end{proof}

\begin{Remark}\label{rem:whitney-vs-disjunction}
\begin{enumerate}[label=\textup{(\alph*)}]
\item Lemmas and Lemma \ref{t:gdi} are generalizations, to $(r-1)$-multiplicity and to codimension 2, of the `high-connectivity' version of the Whitney trick \cite{Ha63}, \cite[Lemma 4.2]{Ha84}, \cite[Theorem 4.5 and appendix A]{HK}, \cite[Theorem 4.7 and appendix]{CRS}.%
\red{\footnote{\red{As the proof shows, \cite[Theorem 4.5]{HK} and \cite[Theorem 4.7]{CRS} are correct for $s=-1$.
In those results the $(s+1)$-connectivity of the self-intersection set $\widetilde\Delta$ for some reason was replaced by the $(s+1)$-connectivity of the classifying map $\Delta\to\R P^\infty$.
In the proof instead of the homotopy triviality of the composition $S^i\to\Delta\to\R P^\infty$ one can either use
the triviality of a line bundle $\lambda_\Delta|_{S^i}$ for $i>1$, or the fact that the immersion
$\widetilde\Delta\to M^m$ is framed.
Note that $N(\widetilde\Delta,D^n)$ is not defined because $\widetilde\Delta\not\subset D^n$.}}}
The lemmas are proved by ambient surgery, i.e. by first adding to $f\sigma_{r-1}$ `an embedded 1-handle' along a path joining $x$ to $y$ in $f\sigma_1\cap\ldots\cap f\sigma_{r-1}$ (which is assumed by induction to be already connected), and then cancelling `an embedded 2-handle' along the `Whitney disk', which for codimension $\ge3$ was done in \cite[\S3]{Ha62} ($r=2$), \cite[proof of Theorem 1.1 in p. 7]{Me} ($r=3$).

\red{An analogous remark holds for a possible direct proof (which we do not present) of the Local Disjunction Theorem \ref{l:ld+3} and $\sigma_1,\ldots,\sigma_{r-1}$ replaced by $D_1,\ldots, D_{r-1}$.}

For a generalization to the `metastable' version see \cite{MW', MW16, Sk17}\red{, \cite{Sk'}}.

\item Applying the Disjunction Theorems in the form presented here may introduce new $r$-fold points (albeit no global ones). On the other hand, for $k\geq 3$, the higher-multiplicity Whitney trick in \cite[Theorem~17]{MW} does not create any new $r$-fold points at all. This difference is immaterial for the study of almost $r$-embeddings or {\it ornaments} (see \S\ref{s:ornsta}),
but it is important in for the study of {\it doodles} (see \S\ref{s:ornsta}).

For $k\geq 3$, our proof can perhaps be modified to show that in the Local Disjunction Theorem \ref{l:ld+3}, under the additional assumption that   $f$ embeds each disk, we may obtain additionally that the resulting map $f'$
embeds each disk, as in \cite[Theorem~17]{MW}. Such an improvement might be obtained by an application of the corresponding (known) `injective' version of the Surgery of Intersection Lemma \ref{l:surg}.

If $k=r=2$, we cannot obtain this (as e.g. disks extending the Whitehead link show).
It would be interesting to know if we can obtain this for $k=2$, $r\ge3$.

\end{enumerate}
\end{Remark}

\subsection{The Singular Borromean Rings and proof of Theorem~\ref{t:example-2-4}}\label{s:FKT-example}

We consider the following lemma (required for Theorem \ref{t:example-2-4}) interesting in itself.

\begin{Lemma}[Singular Borromean Rings]\label{l:bor}
For each $n=2l$ let  $T:=S^l\times S^l$ be the $2l$-dimensional torus with meridian $m:=S^l\times \cdot$ and parallel $p:=\cdot\times S^l$, and let $S^n_p$ and $S^n_m$ be copies of $S^n$.
Then there is no PL map $f\colon T\sqcup S^n_p\sqcup S^n_m\to \R^{n+l+1}$ satisfying the following three properties: 
\begin{enumerate}[label=\textup{(\roman*)}]
\item the $f$-images of the components are pairwise disjoint;

\item $fS^n_p$ is linked modulo 2 with $fp$ and is not linked modulo 2 with $fm$,
\footnote{See the well-known definition of `linked modulo 2' e.g. in~\cite[\S77]{SeTh80} or in~\cite[\S4.2 `Linking modulo 2 of curves in space']{Sk}.}
 and

\item $fS^n_m$ is linked modulo 2 with $fm$ and is not linked modulo 2 with $fp$.
\end{enumerate}
\end{Lemma}

\begin{proof}
The proof uses a `triple intersection' homology argument analogous to the classical proof showing that Borromean rings are linked \cite{Po}, \cite[\S4.5 `Massey-Milnor number modulo 2']{Sk}.
The reader might want to read this proof first for
$n=2$ and $l=1$.

Assume to the contrary that the map $f$ exists.
Without loss of generality, we may assume that $f$ is in general position.

Throughout the proof all the chains and cycles are assumed to have $\Z_2$ coefficients, and all the equalities are congruences modulo 2.
Since all the chains below are represented by general position polyhedra, chains could be identified with their supporting bodies.
We denote by $\partial$ the boundary of a chain.

We can view $f(T)$, $f(S^n_p)$, and $f(S^n_m)$ as $2l$-, $n$- and $n$-dimensional PL cycles in general position in $\R^{n+l+1}$.
Denote by $C_T$, $C_p$, and $C_m$ singular cones in general position over $f(T)$, $f(S^n_p)$, and $f(S^n_m)$, respectively.
We view these cones as $(2l+1)$-, $(n+1)$- and $(n+1)$-dimensional PL chains.
The contradiction is
$$0\underset{(1)}=|\partial(C_T \cap C_p \cap C_m)|\ \underset{(2)}=
\ |\underbrace{\partial C_T}_{=f(T)} \cap C_p \cap C_m|\ +
\ |C_T \cap \underbrace{\partial C_p}_{=f(S^n_p)}\cap C_m|\ +
\ |C_T \cap  C_p \cap \underbrace{\partial C_m}_{=f(S^n_m)}|\ \underset{(3)}=1+0+0=1.$$
Here (1) follows because $C_T \cap C_p \cap C_m$ is a $1$-dimensional PL chain, so its boundary is $0$.
Equation (2) is Leibniz formula.
So it remains to prove (3).
\phantom\qedhere
\end{proof}

\begin{proof}[Proof of (3)]
For $X\in \{T,S^n_m,S^n_p\}$ denote $f_X:=f|_X$.

For the {\it second term} we have
$$|C_T \cap f(S^n_p) \cap C_m|
\ \overset{(*)}=\ |(f_{S^n_p}^{-1}C_T)\cap(f_{S^n_p}^{-1}C_m)|\ \overset{(**)}=\ 0,\quad\text{where}$$
\begin{itemize}
\item[(*)] holds because $(n+1)+(2l+1)+2n<3(n+l+1)$, so by general position $C_T\cap C_m$ avoids self-intersection points of $f(S^n_p)$,

\item[(**)] holds by the well-known higher-dimensional analogue of \cite[Parity Lemma 3.2.c]{Sk14} (which is proved analogously) because the intersecting objects are general position cycles in $S^n_p$; they are cycles because $\partial(C_T \cap f(S^n_p))=0=\partial(C_m \cap f(S^n_p))$ and
$n\le2l\Leftrightarrow(n+1)+2n<2(n+l+1)$, so by general position both $C_T$ and $C_m$ avoid self-intersection points of $f(S^n_p)$.
\end{itemize}
Analogously $|C_T \cap  C_p \cap f(S^n_m)|=0$.

For the {\it first term} we have
$$|f(T) \cap C_p \cap C_m|
\ \overset{(***)}=\ |(f_T^{-1}C_p)\cap (f_T^{-1}C_m)|\ \overset{(****)}=\ m\cap p\ =\ 1,\quad\text{where}$$
\begin{itemize}
\item[(***)] holds because $n\ge l\Leftrightarrow 2(n+1)+4l<3(n+l+1)$, so by general position $C_p\cap C_m$ avoids self-intersection points of $f(T)$,

\item[(****)] is proved as follows:

The $l$-chain $f_T^{-1}C_p$ is a cycle in $T$ because $\partial(C_p \cap f(T))=0$ and
$n\ge2l\Leftrightarrow n+1+4l<2(n+l+1)$, so by general position $C_p$ avoids self-intersection points of $f(T)$.
By conditions (b) and (c) of Lemma~\ref{l:bor} we have
$$|p\cap f_T^{-1}C_p|\ =\ |f(p)\cap C_p|\ =\ 1 \text{\quad and \quad}
|m\cap f_T^{-1}C_p|\ =\ |f(m)\cap C_p|\ =\ 0.$$
I.e. the cycle $f_T^{-1}C_p$  intersects the parallel $p$ and the meridian $m$ at $1$ and $0$ points modulo $2$, respectively.
Therefore $f_T^{-1}C_p$ is homologous to the meridian $m$.
Likewise, $f_T^{-1}C_m$ is homologous to the parallel $p$.
This implies (****).
\end{itemize}
\end{proof}

\noindent\emph{Construction of the $2$-complex in Theorem~\ref{t:example-2-4}.}
We begin by recalling the construction of the $2$-complex $K$ from \cite{FKT}.
Let $P$ be the $2$-skeleton of the $6$-simplex whose vertices are $\{p_1,\dots,p_7\}$.
Let $p:=\boundary [p_1,p_2,p_3]$ denote the boundary of the $2$-simplex $[p_1,p_2,p_3]$.
Denote by $P_-$ the complement in $P$ to (the interior of) the $2$-simplex $[p_1,p_2,p_3]$.
The remaining four vertices $p_4,p_5,p_6,p_7$ span a `complementary' $2$-sphere $S^2_p:=\partial [p_4,p_5,p_6,p_7]  \subset P$ that is the boundary of the $3$-simplex $[p_4,p_5,p_6,p_7]$ (this 3-simplex itself is not contained in $P$).

Let $M_-$ denote a copy of $P_-$ on a disjoint set of vertices $\{m_1,m_2\dots,m_7\}$, and let $m:=\partial[m_1,m_2,m_3]$ and
$S^2_m:=\partial [m_4,m_5,m_6,m_7]$.

The 2-complex $K$ then is defined by the formula
$$K:=(P_-\underset{p_1=m_1}\cup M_-)\underset{p=S^1\times\cdot,\ m=\cdot\times S^1}{\cup} T,$$
where $T$ is the torus $S^1\times S^1$ with any triangulation for which $S^1\times\cdot$, $m=\cdot\times S^1$ are subcomplexes.

\begin{Lemma}\label{l:VK-complex}\cite[Satz~5]{VK}
Let $g\colon P\to \R^4$ be a PL map in general position of the 2-skeleton of the 7-simplex.
Then the number $v(g)$ of intersection points of $f$-images of disjoint triangles (i.e., the total number of global $2$-fold points of $g$) is odd.
\footnote{The lemma implies that the Van Kampen obstruction of $P$ is nonzero even modulo $2$, or equivalently, $P$ does not admit a `$\Z_2$-almost $2$-embedding' in $\R^4$.
For an elementary exposition and an alternative proof see \cite{Sk14}.}
\end{Lemma}

\begin{proof}[Proof of Theorem~\ref{t:example-2-4}]
Analogously to \cite[\S3.3]{FKT}, $K$ admits a $\Z$-almost $2$-embedding in $\R^4$.
Suppose to the contrary that there is a PL almost $2$-embedding $f\colon K\to \R^4$.
We may assume it is in general position.
Let us show that $f|_{S^2_p\sqcup S^2_m \sqcup T}$ satisfies the conditions (a), (b) and (c) of
the Singular Borromean Rings Lemma~\ref{l:bor} (this is essentially proved in \cite[Lemma 6]{FKT}).
This would give a contradiction by Lemma \ref{l:bor}.

Condition (a) is satisfied because $f$ is an almost $2$-embedding and because any
simplex in the triangulation of $T$ is vertex-disjoint from any simplex in $S^2_p$ and
from any simplex in $S^2_m$.

The complex $K$ contains the cone $p_4\ast p$, which is a disk disjoint from $S^2_m$.
Since $f$ is an almost 2-embedding, $f(p_4\ast p)\cap f(S^2_m)=\emptyset$.
Then $f(p)$ and $f(S^2_m)$ are unlinked modulo 2.
Analogously, $f(m)$ and $f(S^2_p)$ are unlinked modulo 2.

Extend $f|_{P_-}$ to a general position PL map $g:P\to \R^4$.
Then the sphere $f(S^2_p)=g(S^2_p)$ and the circle $f(p)=g(p)$ are linked modulo 2 because
$$|g(S^2_p) \cap g[p_1,p_2,p_3]|\ =
\ \sum_{\{i,j,k\} \subset \{4,5,6,7\}} |g[p_i,p_j,p_k] \cap g[p_1,p_2,p_3]|\ \overset{(1)}\equiv\ v(g)
\ \overset{(2)}=\ 1\ \in\ \Z_2,\quad\text{where}$$
\begin{itemize}
\item[(1)] holds because $f|_{P_-}$ is an almost 2-embedding, so $f(\sigma)\cap f(\tau)=\emptyset$ for all `other'  pairs $\sigma,\tau$;

\item[(2)] holds by Lemma~\ref{l:VK-complex}.
\end{itemize}
Analogously the sphere $f(S^2_p)=g(S^2_p)$ and the circle $f(p)=g(p)$ are linked modulo 2.
\end{proof}

\subsection{Proof of Theorems~\ref{c:nld1} and \ref{thm_correspondance3}.a}\label{s:ornpro}

\begin{proof}[Proof of Theorem~\ref{thm_correspondance3}.a]
The case $r=k=2$ is known, cf. Remark \ref{r:hist}.c.
We present the proof for $r=3$, the generalization to arbitrary $r\ge3$ or to $r=2\le k-1$ is obvious
(because by the Global Disjunction Theorem~\ref{t:elim}.(a)-(b) assertion $(D_{k,r})$ is true for arbitrary $k\ge2$ and $k+r\ge5$).

{\it We first prove surjectivity}, i.e., that for any integer $l$ there is an ornament (actually a doodle) $f$ such that $\lk f = l$.

The case $l=0$ is trivial, we can take any doodle such that the images of its connected components lie in $3$ pairwise disjoint balls.

Consider now the case $l=\pm 1$.
Identify $B^{3k}$ with $B^k\times B^k \times B^k$.
Define the {\it Borromean ornament (doodle)} $f:\bigsqcup_{i=1}^3 S_i^{2k-1}\rightarrow S^{3k-1}=\partial B^{3k}$ by
$$fS_1^{2k-1}=\partial(B^k\times B^k\times \cdot),\quad fS_2^{2k-1}=\partial(B^k\times \cdot \times B^k),
\quad\text{and}\quad fS_3^{2k-1}=\partial(\cdot \times B^k\times B^k).$$
Clearly, $|\lk f|=1$.
By composing $f$ with the reflection of one of the spheres $S_i^{2k-1}$ we get a new ornament $f'$ such that
$\lk f'=-\lk f$.
So $\{\lk f, \lk f'\}=\{-1,1\}$.

Let $f_0,f_1$ be two ornaments, their images lying in disjoint balls.
Connect each of the connected components of $f_0$ with the respective connected component of $f_1$ by a thin tube
and denote the obtained doodle by $f_2$.
Clearly, $\lk f_2=\lk f_0+\lk f_1$.
So the case of general $l$ follows from the cases $l=0$ and $l=\pm 1$.

{\it We now prove injectivity.}
We have to prove that if $f_0,f_1 :  \bigsqcup_{i=1}^3 S_i^{2k-1} \rightarrow S^{3k-1}$ are two ornaments such that
$\lk f_0 = \lk f_1$, then $f_0$ and $f_1$ are ornament concordant.

Take a general position PL map
$F :  (\sqcup_{i=1}^3 S_i^{2k-1}) \times I \rightarrow S^{3k-1} \times I$
such that
$
F(\cdot , 0) = f_0(\cdot) \times  0$  and  $F(\cdot,1) = f_1(\cdot) \times 1.
$
Since $\lk f_0 = \lk f_1$, the set
$
F(S_1^{2k-1} \times I) \cap F(S_2^{2k-1} \times I) \cap F(S_3^{2k-1} \times I)
$
consists of pairs of $3$-fold points of opposite signs.
Each such pair can be eliminated by the Global Disjunction Theorem~\ref{t:elim}.(a)-(b) applied to
$K=\bigsqcup_{i=1}^3 S_i^{2k-1}\times I$.
\end{proof}

\begin{figure}[h]
\centerline{\includegraphics[width=6cm]{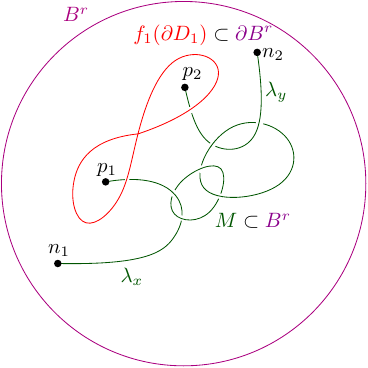} \qquad
\includegraphics[width=6cm]{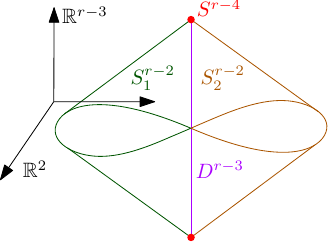} }
\caption{(a) To Lemma \ref{l:cdm1}.
(b) To the proof of Lemma \ref{l:cdm1}.}
\label{f:join}
\end{figure}

\begin{Lemma}\label{l:cdm1}
For each $r\ge3$ there is a proper general position PL map $f:\partial D_1\sqcup D_2\sqcup\ldots\sqcup D_r\to B^{r}$, where $D_j$ is a copy of $(r-1)$-disk, such that
\begin{enumerate}
\item $M:=fD_2\cap\ldots\cap fD_r$ is a proper oriented submanifold of $B^r$ and $\partial M=\{p_1,p_2,n_1,n_2\}\subset \partial B^r$, where the points $p_1,p_2$ have positive sign and the points $n_1,n_2$ have negative sign (the signs are defined as the signs of intersection points of $r-1$ oriented $(r-2)$-dimensional spheres in $S^{r-1}$),

\item for any generic oriented path $\lambda$ in $B^r$ from $p_j$ to $n_i$ and any proper extension $g:D_1\sqcup D_2\sqcup\ldots\sqcup D_r\to B^r$ of $f$ we have $gD_1\iprod\lambda=(-1)^j$.
\end{enumerate}
\end{Lemma}

\begin{proof}
It is easy to define the map $f$ on $D_2\sqcup\ldots\sqcup D_r$ so that the property $(1)$ is satisfied.
Let us now define $f$ on $\partial D_1$.

Identify $S^{r-1}=\partial B^r$ with $S^2 * S^{r-4}$, and $S^{r-2}=\partial D_1$ with $S^1*S^{r-4}$.
(This works for $r=3$, when $S^{r-4}=\emptyset$.)
Without the loss of generality we may assume that $\{ p_1,p_2, n_1,n_2\}\subset S^2*\emptyset\subset S^{r-1}$.
Let $8:S^1\rightarrow S^2$ be a map whose image is figure ``8'' winding $1$ time around $p_1$, $-1$ time around $p_2$, and $0$ times around $n_1$ and $n_2$, i.e. $\lk(8,n_i)=0$ and $\lk(8,p_j)=(-1)^{j+1}$.
Now define $f|_{\partial D_1}:\partial D_1\rightarrow S^{r-1}$ by $f:=8*\id S^{r-4}$ (see Figure \ref{f:join}.b).

Let us prove that $f$ satisfies $(2)$.
Let $\pi$ be a generic oriented path in $S^2=S^2*\emptyset\subset S^{r-1}$ from $p_j$ to $n_i$.
Then
$$
gD_1\iprod \lambda=gD_1\iprod(\lambda\cup-\pi) + gD_1\iprod \pi =
0 + f\partial D_1\iprod \pi = 8\iprod \pi =\lk(8,n_i-p_j)=0-(-1)^{j+1}=(-1)^j.
$$
\end{proof}

\begin{proof}[Proof of Theorem~\ref{c:nld1}]
Let us prove that for each map $f$ given by Lemma~\ref{l:cdm1} the ornament $f|_{\partial D}$ is as required.
Extend $f$ to $D_1$ properly and generically in an arbitrary way (e.g., by coning over a generic point).

{\it Proof that $f D_1 \iprod \ldots \iprod f D_r = 0$.}
By the property $(1)$ of Lemma~\ref{l:cdm1} (and possibly by exchanging $n_1,n_2$)
we may assume without the loss of generality that $M$ consists of
generic oriented paths $\lambda_j$ from $p_j$ to $n_j$, $j=1,2$, and
a union $\omega$ of disjoint embedded circles (see Figure \ref{f:join}.a).
Then by the property $(2)$ of Lemma~\ref{l:cdm1} we have
$f D_1 \iprod \ldots \iprod f D_r = f D_1\iprod (\lambda_1\sqcup \lambda_2 \sqcup \omega) = -1+1+0 = 0$.

{\it Proof that $g D_1 \cap \ldots \cap g D_r \neq \emptyset$ for any other proper generic map
$g \colon D \to B^r$ such that $f=g$ on $\partial D$.}
Since $f=g$ is generic on the boundary, we have that $M':=gD_2\cap\ldots\cap gD_r$ is a relative $1$-dimensional integer homology cycle in $B^r$ and $\partial M'=\{ p_1,p_2, n_1,n_2\}$.
Without the loss of generality (and possibly by exchanging $n_1,n_2$),
we may assume that $M'$ contains an oriented path $\lambda_1$ from $p_1$ to $n_1$.
By the property $(2)$ of Lemma~\ref{l:cdm1}, $p_1$ and $n_1$ are in the different connected components of
$B^r\setminus gD_1$.
So
$$\emptyset\neq gD_1\cap \lambda_1\subset gD_1\cap M'=g D_1 \cap \ldots \cap g D_r.$$
\end{proof}

\section{Discussion and open problems}\label{s:disc}

\begin{Remark}
\label{r:hist}

\begin{enumerate}[label=\textup{(\alph*)}]
\item Analo\-gous\-ly to \cite[\S5]{MW}, it can perhaps be shown that the analogue of Theorem \ref{t:tve} holds for
$d\geq 2r$, by using the results in the present paper (the Global Disjunction Theorem~\ref{t:elim} below)
to rewrite the proofs of \cite[Thm.~11]{MW} with $k \ge 2$ instead of $k\ge3$.
Since the necessary facts about prismatic maps are not gathered in one easily citable statement in the current
version of \cite[\S5]{MW} but dispersed throughout the text, for simplicity of presentation we focus here on the shorter argument for $d\geq 2r+1$.

\item Theorem \ref{t:z-alm3} for $r=2$ was a step in the proof of a classical algebraic criterion of van Kampen, Shapiro and Wu for embeddability of simplicial $n$-complexes into $\R^{2n}$ \cite{VK, Sh57, Wu65}, see survey \cite[Theorem 4.1]{Sk08}.
Both this criterion and Theorem \ref{t:z-alm3} for $r=2$ were generalized
by Haefliger and Weber who showed that an $n$-complex $K$ embeds into $\R^d$ iff there is a $\Z_2$-equivariant map from the deleted product $K^{\times2}_\Delta$ to $S^{d-1}$, provided $d\ge3(n+1)/2$ \cite{Ha63, We67}, see survey
\cite[Theorem 8.1 and Proposition 8.4]{Sk08}.
\red{The dimension restriction is related to the existence of $n$-dimensional Borromean rings in $\R^d$ for $d<3(n+1)/2$, i.e. of an embedding $S^n\sqcup S^n\sqcup S^n\to \R^d$ whose pairwise linking coefficients are zeroes but which is not isotopic to the standard embedding.
(For the definition of the linking coefficient for $d\leq 2n$, see, e.g., \cite[\S3]{Sk08}.)}
One might conjecture that the dimension restriction $d\ge3(n+1)/2$ can be weakened if one uses the configuration space of {\it $r$-tuples} of distinct points, and that methods of this paper would allow to prove such a conjecture.
Surprisingly, this is not so, see \cite[end of \S5]{Sk08}, \cite[end of \S1]{CS}.
An explanation is that the notion of an embedding is more subtle than the notion of almost $r$-embedding.

On the other hand, for a generalization of Theorem \ref{t:z-alm3} to $n$-complexes in $\R^d$ keeping $r$ arbitrary
see \cite[Theorem~2]{MW'}, \cite[Theorem 2]{MW16}, \cite[Theorems 1.1-1.3]{Sk17}\red{, \cite{Sk'}}.

\item The case $r=2$, $k\ge3$ of the Local Disjunction Theorem \ref{l:ld+3} is a version of the Whitney trick; the subcase $k=2$ is an exercise on elementary link theory.
Here is a well-known proof for $r=k=2$ for the general position case when $f|_{\partial D}$ is an embedding.
Given a 2-component 1-dimensional link in $S^3$, one can unknot one component in the complement of the second by crossing changes (or by finger moves, guided along arcs) \cite[Theorem 3.8]{PS96}.
By the assumption the linking number is zero.
The linking number is preserved under crossing changes.
So after crossing changes we obtain a link formed by the unknot and the component
which shrinks in the complement of the unknot.
For such link the assertion is trivial.

\red{The case $r=k=2$ of the Local Disjunction Theorem \ref{l:ld+3} can also be proved using Casson's finger
moves analogously to \cite[Proof of Disjunction Lemma 2.1.b]{Sk00}, and is clear when $f|_{D_1}$ is an unknotted embedding.}

%

\item In \cite{SSS} it is shown that
{\it for each $(n,d)$ such that $n+2\le d\le \frac{3n}2+1$ there exists a finite $n$-dimensional complex $K$ that admits an almost $2$-embedding in $\R^d$ but that does not embed into $\R^d$.}
(This example was used to show, for such $n,d$, the incompleteness of deleted product obstruction, which is defined before Proposition \ref{cor:equiv-alm}.)
For $d=2n=4$ this improves \cite{FKT} in a different direction than Theorem \ref{t:example-2-4}: {\it there exists a finite $2$-dimensional complex $K$ that admits an almost $2$-embedding in $\R^4$ but that does not embed into $\R^4$.}
\end{enumerate}
\end{Remark}

\begin{Remark}
\label{rem:ornaments}
\begin{enumerate}[label=\textup{(\alph*)}]

\item  Assume that $(d,n,r)=(2,1,3)$ (hence $2d=3n+1$).
In this case, a triviality criterion for ornaments is given in \cite{Me03};
it would be interesting to know if it is algorithmic and if it extends to a classification.
The $r$-linking number is not a complete invariant for doodles, e.g., there is a non-trivial $(2,1,3)$-doodle
with zero $3$-linking number.%
\footnote{This is written in \cite[bottom of p. 39 and fig.4]{FT} with a reference to a later paper.
It would be interesting to have a published proof.
It might be easier to obtain the proof using the `intersection' language,
see \S\ref{s:FKT-example}, rather than `commutators' language \cite{FT}.}
Thus the analogue of Theorem \ref{thm_correspondance3}.b for $k=1$ and $r=3$ is false.
(We conjecture that such an analogue is also false for $k=1$ and each $r\ge4$, cf. Theorem \ref{c:nld1}.)
See \cite{B} for a study of ornaments which are PL immersions, up to regular ornament homotopy
(they were called doodles, which is different from terminology of this paper).

\item There is a concordance version of Theorem \ref{t:genorn} for $p$-ornaments in $S^{kr-1}$, where
$p:=k(r-1)-1$. It involves a complete invariant in $H^p(K_1;\Z)\times\ldots\times H^p(K_r;\Z)$.

\item An {\it $s$-component $r$-multiplicity ornament in $S^d$} is a PL general position map
$f:K_1\sqcup\ldots\sqcup K_s\to S^d$ of disjoint union of $s$ simplicial complexes such that the intersection of any $r$ objects among $fK_1,\ldots,fK_s$ is empty.
Although here we only consider the case $s=r$, our results have straightforward generalizations to $s>r$.
2-multiplicity ornaments were widely studied under the name of {\it link maps}, mostly for the case when each $K_j$ is a sphere, see \cite{Sk00} and references therein.
\end{enumerate}
\end{Remark}

\begin{Remark}
\label{rem:bor1}
\begin{enumerate}[label=\textup{(\alph*)}]
\item Our proof of the Singular Borromean Rings Lemma \ref{l:bor} for $n=2$ and $l=1$ gives a shorter,
elementary proof of the result from \cite{FKT} mentioned before Theorem \ref{t:example-2-4}.

\item In \cite{ST}, the Singular Borromean Rings Lemma \ref{l:bor} is used to study algorithmic aspects of almost 2-embeddability of complexes in $\R^d$.

\item The analogue of the Singular Borromean Rings Lemma~\ref{l:bor} for $n=l+1=1$ is true, although our proof does not work for this case.
The analogue of Lemma~\ref{l:bor} for $n=l$ is false, but would conjecturally become true if we add an additional condition that $f(S^n_p)$ and $f(S^n_m)$ are unlinked modulo 2.
For the corresponding construction of Borromean rings in $\R^3$ see \cite[\S4.4 `Borromean rings and commutators']{Sk}.


\item Lemma \ref{l:bor} for $n=2$, $l=1$ and \emph{embedded} torus $f(T)$ was proved in
\cite[Theorem 1 and the middle paragraph on page 53]{KT} (in a much more general form).
We are grateful to S. Krushkal and P. Teichner for explanation of how the proof of \cite{KT} works for the case of \emph{non-embedded} torus, as well as for sketching a short direct proof of Lemma \ref{l:bor} for $n=2$, $l=1$ involving the Milnor group of the complement.
\red{It would be nice if these arguments were publicly available.}
It would be interesting to know if these arguments can be generalized to higher dimensions.

\red{Also, Lemma \ref{l:bor} for \emph{embedded} spheres $f(S^n_p)$ and $f(S^n_m)$ was (not stated but) essentially proved in \cite{FKT, SSS}, cf. \cite[Borromean rings Lemma 4.4.3]{Sk}.
It would be interesting to know whether the proof of \cite{SSS} extends to the case of non-embedded spheres.
The proof in \cite{FKT} uses the fact that $f(S^2_p)$ and $f(S^2_m)$ are \emph{embedded} to deduce by Alexander duality that $H_2(\R^4 - f(S^2_p \sqcup S^2_m))=0$, and then applies the Stallings Theorem on the lower central series of groups \cite[Proof of Lemma~7]{FKT}.
This may fail if $f(S^2_p)$ and $f(S^2_m)$ are not embedded.

Comparing our proof with \cite{FKT, KT} illustrates in a geometric (more precisely, homological)
language the relation between Massey products and  commutators, cf. \cite{Po}.}


\end{enumerate}
\end{Remark}


\begin{Remark}[Open problems]
\label{rem:open}
\begin{enumerate}[label=\textup{(\alph*)}]
\item Does the analogue of Theorem \ref{t:z-alm3} holds for $k=1$ and large enough $r$?
Cf.\ Theorems \ref{c:nld1} and \ref{t:elim}.d.

\item Is there an example for Theorem \ref{c:nld1} for which $f|_{\partial D}$ is an embedding?

\item Does the analogue of Theorem~\ref{thm_correspondance3}.b hold for $k=2$? In Theorem~\ref{thm_correspondance3}, can ornament [doodle] concordance be replaced by ornament [doodle] homotopy? Here, an ornament [doodle] concordance $F$ is an ornament [doodle] homotopy if it is `level preserving', i.e., if $F(\cdot,t) \subset S^m \times \{ t \}\quad\text{for each}\quad t\in I.$

\item {\bf Gromov's problem} \cite[2.9.c]{Gr}.
{\it Is it correct that if
$r$ is not a prime power, then for each compact subset $K$ of $\R^m$ for some $m$, having Lebesgue dimension
$\dim K=(r-1)k$, there is a continuous map $X\to\R^{kr}$ each of whose point preimages contains less than $r$ points?}

The analogue of this problem for polyhedra $K$ and almost $r$-embeddings instead of maps without $r$-fold points is true by Theorem \ref{c:all}.a.
\red{We suspect that the answer is `no'.
If $K$ is a finite simplicial complex and one requires the map  $K\to\R^{kr}$ to be PL instead of continuous,
the question seems to be more complicated.}

\item Let $X$ be a compact subset of $\R^m$ for some $m$.
Is it correct that $\dim(X\times X\times X)<6n$ if and only if any continuous map $X\to\R^{3n}$ can be arbitrary close approximated by a continuous map without triple points?
This is interesting for `fractal' $2n$-dimensional compacta $X$, for which
$\dim(X\times X\times X)<3\dim X$.


\red{\item Which of the results of \cite{DRS, ST91, RS98, Sk00}, \cite[Disjunction Theorem 3.1]{Sk02}, \cite[Theorems 4.4, 5.4, 5.5, Example 5.9.c]{Sk08} on $2$-fold intersections can be generalized to $r$-fold intersections?}
\end{enumerate}
\end{Remark}

\end{document}